\documentclass[12pt]{amsart}

\usepackage{amssymb,amsmath}
\usepackage{amssymb,latexsym}
\usepackage{amsfonts}
\usepackage{amssymb}
\usepackage{stmaryrd}
\usepackage{longtable}
\usepackage{amsmath, geometry, amssymb}
\usepackage{chngcntr}
\counterwithin{table}{section}
\numberwithin{equation}{section}
\newtheorem{theorem}{Theorem}[section]
\newtheorem{lemma}{Lemma}[section]
\newtheorem{coro}{Corollary}[section]

\newtheorem{example}{Example}[section]

\newcommand{\sqr}[2]{{\vcenter{\vbox{\hrule height#2pt
                \hbox{\vrule width#2pt height#1pt \kern#1pt
                \vrule width#2pt}\hrule height#2pt}}}}

\newcommand{\beq}{\begin{equation}}
\newcommand{\eeq}{\end{equation}}
\newcommand{\beqar}{\begin{eqnarray}}
\newcommand{\eeqar}{\end{eqnarray}}
\def\beqars{\begin{eqnarray*}}
\def\eeqars{\end{eqnarray*}}

\def \ds{\displaystyle}

\newcommand{\nn}{\mathbb{N}}
\newcommand{\zz}{\mathbb{Z}}
\newcommand{\qq}{\mathbb{Q}}
\newcommand{\cc}{\mathbb{C}}

\newcommand{\DD}{\Delta}

\makeatletter
\def\imod#1{\allowbreak\mkern5mu({\operator@font mod}\,\,#1)}
\makeatother

\allowdisplaybreaks

\begin{document}

\title{Explicit Evaluation of Double Gauss Sums}
\author{\c{S}aban Alaca and Greg Doyle}

\maketitle

\markboth{\c{S}ABAN ALACA AND GREG DOYLE}{EXPLICIT EVALUATION OF DOUBLE GAUSS SUMS}

\begin{abstract}
We present an explicit evaluation of the double Gauss sum
$$ G(a,b,c;S;p^n):=\ds\sum_{x,y=0}^{p^n-1} e^{2\pi i S(ax^2+bxy+cy^2)/p^n},$$
where  $a, b, c$ are integers such that $\gcd(a,b,c)=1$,  $p$ is a prime, $n$ is a positive integer,  and $S$ is an integer coprime to $p$.

\vspace{2mm}


\vspace{2mm}

\end{abstract}

\section{Introduction}  

We let $\nn$ denote the set of positive integers, $\nn_0$ the non-negative integers, $\zz$ the integers, $\qq$ the rational numbers and $\cc$ the complex numbers.
Let $a,b,c,S \in \zz$, $n \in \nn$ and let $p$ be a prime.
For convenience we set $\ds e(\alpha) := e^{2 \pi i \alpha}$ for any $\alpha \in \qq$.

The (quadratic) Gauss sum $\ds G(S;p^n)$ is defined by
    \begin{align}
        G(S;p^n) := \sum_{x=0}^{p^n-1} e\left(\frac{Sx^2}{p^n}\right).
    \end{align}
The evaluation of $G(S;p^n)$ is well known and was first determined by Gauss \cite{Gauss}.
In this paper we evaluate a similar sum with binary quadratic form argument.

We define the double (quadratic) Gauss sum $G(a,b,c;S;p^n)$ by
    \begin{align}
        G(a,b,c;S;p^n) := \sum_{x,y=0}^{p^n-1} e\left(\frac{S(ax^2+bxy+cy^2)}{p^n}\right).
    \end{align}

If $p^n \mid S$ then we have
\begin{align}
G(S;p^n)= \sum_{x=0}^{p^n-1} 1 =p^n \quad \text{ and }\quad G(a,b,c;S;p^n)= \sum_{x,y=0}^{p^n-1} 1 =p^{2n}.
\end{align}
If $p^n \nmid S$ then  $p^m || S$ for some $m \in \nn_0$ with $m <n$, so that   $S_1 = S/p^m \in \zz$, and
\begin{align}
G(S;p^n) &=  \sum_{x=0}^{p^n -1} e\left(\frac{Sx^2}{p^n}\right)
=  \sum_{x=0}^{p^n -1} e\left(\frac{S_1 x^2}{p^{n-m}}\right)    \\
&=  p^m \sum_{x=0}^{p^{n-m} -1} e\left(\frac{S_1 x^2}{p^{n-m}}\right) \nonumber
=p^m G(S_1;p^{n-m})
\end{align}
and 
\begin{align}
G(a,b,c;S;p^n) &= \sum_{x,y=0}^{p^n-1} e\left(\frac{S(ax^2+bxy+cy^2)}{p^n}\right)
 = \sum_{x,y=0}^{p^n-1} e\left(\frac{S_1(ax^2+bxy+cy^2)}{p^{n-m}}\right)  \nonumber \\
&= p^{2m} \sum_{x,y=0}^{p^{n-m}-1} e\left(\frac{S_1(ax^2+bxy+cy^2)}{p^{n-m}}\right)  \\
&=p^{2m}G(a,b,c;S_1;p^{n-m}).  \nonumber
\end{align}
Thus
we  may assume that $S$ is coprime to $p$ and  $\gcd(a,b,c)=1$.
We write $(a,b,c)$ to denote   $\gcd(a,b,c)$.
The sum $G(a,b,c;S;p^n)$  was first evaluated by Weber \cite{Weber} for $b$ even, and subsequently refined by Jordan \cite{Jordan} around $1870$. Alaca, Alaca and Williams \cite{AAW} evaluated the sum $G(a,b,c;S;p^n)$ given the condition $\ds 4ac-b^2 \neq 0$.
In each of these papers, the main idea is to use a linear change of variables to diagonalize the binary quadratic form
    \begin{align}
        Q := Q(x,y) = ax^2+bxy+cy^2.
    \end{align}
In this fashion, if $Q \sim Q' = Ax^2+Cy^2$ for some integers $A$ and $C$, then
    \begin{align}
        G(Q;S;p^n) = G(Q';S;p^n) = G(AS;p^n) \cdot G(CS;p^n).
    \end{align}

Our approach is similar to those earlier ideas, but our choice for change of variables will generalize those results in an explicit fashion.
Our approach differs from the approach of Weber  \cite{Weber} and Jordan  \cite{Jordan}, who diagonalized $Q$ recursively.
We  first evaluate $G(a,b,c;S;p^n)$ for $p$ an odd prime, and subsequently $G(a,b,c;S;2^n)$.
It should be noted that references to the papers \cite{Weber} and \cite{Jordan} are rare.
As those papers were written in German and French, respectively, this paper may also serve as a modern translation of those ideas.

We may assume that $a, b \neq 0$, and permuting coefficients if necessary, we may also assume that
if $p^m \mid \mid a$ for some $m \in \nn$, then $p^m \mid c$.
We observe that
\begin{align}
G(a,b,c;S;p^n) = G(\overline{a}, \overline{b}, \overline{c}; S; p^n),
\end{align}
where $\ds \overline{a}$,  $\ds \overline{b}$ and  $\ds \overline{c}$ denote the residue classes of $a$, $b$ and $c$ modulo $p^n$, respectively. Hence, we may
identify $a, b$ and $c$ with their positive integer residues modulo $p^n$. 
We write $a \equiv p^{\alpha}A \imod {p^n}$ and $b \equiv p^{\beta}B \imod {p^n}$, where
$\alpha, \beta \in \nn_0$ and  $A,B\in \zz$ satisfy $p \nmid AB$. 
Thus, we may assume that $a, b$ and $c$ are of the form
    \begin{align}
        (a,b,c)=1,~a\equiv p^{\alpha}A \imod {p^n},~b \equiv p^{\beta}B \imod {p^n},~c \equiv 0 \imod {p^{\alpha}},
    \end{align}
where $A, B \in \zz$ satisfy  $p \nmid AB$, and $\alpha , \beta \in \nn_0$. We note that (1.9) implies that at least one of $\alpha, \beta$ is zero. We use the
inequality $\alpha \leq \beta$ to indicate that $\alpha = 0$, and similarly we use $\beta \leq \alpha$ to indicate that $\beta =0$. Finally, the discriminant of the
binary quadratic form $Q$ in (1.6) is defined by $\DD := 4ac-b^2$.

\section{Preliminary Results for Gauss Sums} 

For the remainder of the paper, we let $\ds \left(\frac Sp\right)$ denote the Jacobi-Kronecker-Legendre symbol.
The following theorem is the famous deep formula first given by Gauss \cite{Gauss}.
One can consult the excellent monograph by Berndt, Evans and Williams \cite[pp. 18-28]{BEW} for an elementary proof.

\begin{theorem}  
    Let $k \in \nn$. For $S \in \zz$ coprime to $k$, we have
        \begin{align*}
            G(S;k):=  \sum_{x=0}^{k-1} e\left(\frac{Sx^2}{k}\right) =
    \begin{cases}
                \ds \left(\frac Sk\right)\sqrt{k} &\mbox{if } k \equiv 1 \imod 4\\
                \ds ~0 &\mbox{if } k \equiv 2 \imod 4\\
                \ds \left(\frac Sk\right)\sqrt{-k} &\mbox{if } k \equiv 3 \imod 4\\
                \ds \left(\frac kS\right)\left(1+i^S\right)\sqrt{k} &\mbox{if } k \equiv 0 \imod 4.
            \end{cases}
        \end{align*}
\end{theorem}

The following corollary is a special case of Theorem 2.1 for $k=p^n$.

\begin{coro}  
    Let $p$ be a prime, $n \in \nn$ and $S \in \zz$ coprime to $p$.

If $p$ is odd, we have
        \begin{align*}
            G(S;p^n) = i^{\left(\frac{p^n-1}2\right)^2} \left(\frac Sp\right)^n \sqrt{p^n}.
        \end{align*}

If $p=2$, we have
        \begin{align*}
            G(S;2^n) = \begin{cases}
                0 &\mbox{if } n = 1\\
                \ds (1+i^S)\left(\frac 2S\right)^n \sqrt{2^n} &\mbox{if } n > 1.
            \end{cases}
        \end{align*}
\end{coro}

We now prove some basic results for  the  Gauss sum (1.1) and for  its variations.

\begin{lemma}  
Let $p$ be a prime, $n \in \nn$ and $S \in \zz$ be coprime to $p$.
We identify $S^{-1}$ with the least positive integer residue of the inverse of $S$  modulo $p^n$. Then we have
        \begin{align*}
             G(S^{-1};p^n) = G(S;p^n)    \text{ and }  G(S^2;p^n) = G(1;p^n).
        \end{align*}
\end{lemma}
\begin{proof}
   We first suppose that $p$ is odd. By Corollary 2.1, we have
        \begin{align*}
            G(S^2;p^n) = \left(\frac {S^2}p\right)^n i^{\left(\frac{p^n-1}2\right)^2} \sqrt{p^n} = G(1;p^n).
        \end{align*}
As $\ds \left(\frac {S^{-1}}p\right) = \left(\frac {S^{2}}p\right)\left(\frac{S^{-1}}p\right) = \left(\frac Sp\right)$,
Corollary 2.1 yields $G(S^{-1};p^n) = G(S;p^n)$.

    We now suppose that $p=2$. If $n=1$, then we have $G(S;2)=0=G(S^{-1};2)$. We assume that $n\geq 2$.
As $S$ is odd,  we have $S^2 \equiv 1 \imod 8$, which implies that $S^{-1} \equiv 1 \imod 8$ so that $\ds \left(\frac 2{S^{-1}}\right) = \left(\frac 2S\right)$ and
further that $i^{S^{-1}} = i^{S}$, so by Corollary 2.1, we have $G(S^{-1};2^n) = G(S;2^n)$. By similar reasoning, one can see that $\ds G(S^2;2^n) = G(1;2^n)$.
\end{proof}

\begin{lemma}  
Let $p$ be a prime, $n \in \nn$ and $S \in \zz$ coprime to $p$. Let $\alpha, \beta \in \nn_0$ satisfy $\alpha, \beta \leq n$.

If $p$ is odd, we have
\begin{align*}
\sum_{\substack{x = 0\\x \equiv 0 \imod {p^{\alpha}}}}^{p^n-1} e\left(\frac{S p^{\beta} x^2}{p^n}\right) = \begin{cases}
p^{n-\alpha} &\mbox{if } 2\alpha  + \beta \geq n\\
p^{\beta} \cdot G(S;p^{n-\beta}) &\mbox{if } 2\alpha + \beta \leq n. \end{cases}
\end{align*}

If $p=2$, we have
\begin{align*}
    \sum_{\substack{x=0\\x \equiv 0 \imod {2^\alpha}}}^{2^n-1} e\left(\frac{S2^{\beta} x^2}{2^n}\right) = \begin{cases}
        2^{n-\alpha} &\mbox{if } 2\alpha + \beta\geq n\\
        0 &\mbox{if } 2\alpha+\beta = n-1\\
        2^{\beta}\cdot G(S;2^{n-\beta}) &\mbox{if } 2\alpha+\beta \leq n-2.
        \end{cases}
\end{align*}
\end{lemma}
\begin{proof}

We first suppose that $p$ is odd. As $\alpha \leq n$, we have
\begin{align}
\sum_{\substack{x = 0\\x \equiv 0 \imod {p^\alpha}}}^{p^n-1} e\left(\frac{S p^{\beta} x^2}{p^n}\right) &=\sum_{x=0}^{p^{n-\alpha}-1} e\left(\frac{S p^{\alpha+\beta} x^2}{p^{n-\alpha}}\right) = G(Sp^{\alpha+\beta};p^{n-\alpha}).
\end{align}
If $2\alpha+\beta \geq n$, we see that (2.1) simplifies to $p^{n-\alpha}$. If $2\alpha+\beta \leq n$,  by (1.4) and Corollary 2.1, we have
    \begin{align*}
        G(Sp^{\alpha+\beta};p^{n-\alpha}) &= p^{\alpha+\beta} \cdot G(S;p^{n-\beta-2\alpha}) = p^{\beta}\cdot G(S;p^{n-\beta}),
    \end{align*}
where we have used the fact that $p^{n-\beta-2\alpha} \equiv p^{n-\beta} \imod 4$. We have similar reasoning for $p=2$, with the exception that if $\ds 2\alpha+\beta
=n-1$ then, by (1.4) and Corollary 2.1, we have $\ds G(S2^{\alpha+\beta};2^{n-\alpha}) = 0$ .
\end{proof}

Note that if $p$ is odd and $2\alpha+\beta=n$ then both statements of Lemma 2.2 agree. Indeed, under this assumption, we see from Corollary 2.1 that
    \begin{align*}
        p^{\beta} \cdot G(S;p^{n-\beta}) = p^{n-2\alpha} \cdot G(S;p^{2\alpha}) = p^{n-\alpha}.
    \end{align*}

We emphasize that the following proof of Lemma 2.3 is clearly modeled after the proof of Lemma~ 3.1~ \cite[pp. 145-147]{AAW} in the paper by Alaca, et al.

\begin{lemma} 
Let $p$ be a prime, $n \in \nn$, $w \in \zz$ and $S \in \zz$ coprime to $p$. Let $\alpha, \beta\in \nn_0$ satisfy $\alpha,\beta \leq n$.

If $p$ is odd and $2\alpha +\beta \leq n$, we have
    \begin{align*}
        \sum_{\substack{x=0\\x \equiv w \imod{p^{\alpha}}}}^{p^n-1} e\left(\frac{Sp^{\beta}x^2}{p^n}\right) = \begin{cases}
            0 &\mbox{if } w \not\equiv 0 \imod {p^{\alpha}}\\
            p^{\beta} \cdot G(S;p^{n-\beta}) &\mbox{if } w \equiv 0 \imod {p^{\alpha}}.
        \end{cases}
    \end{align*}

If $p=2$ and $2\alpha+\beta \leq n -2$, we have
    \begin{align*}
        \sum_{\substack{x=0\\x \equiv w \imod{2^{\alpha}}}}^{2^n-1} e\left(\frac{S 2^{\beta}x^2}{2^n}\right) = \begin{cases}
            0 &\mbox{if } w \not\equiv 0 \imod {2^{\alpha}}\\
            2^{\beta} \cdot G(S;2^{n-\beta}) &\mbox{if } w \equiv 0 \imod {2^{\alpha}}.
        \end{cases}
    \end{align*}
\end{lemma}

\begin{proof}
    We first suppose that $p$ is odd and $2\alpha +\beta \leq n$.
If $\ds w \equiv 0 \imod {p^{\alpha}}$, the statement of the lemma is given by Lemma 2.2 and so we may assume that $p^{\alpha} \nmid w$. We have
        \begin{align}
            \sum_{\substack{x=0\\x \equiv w \imod {p^{\alpha}}}}^{p^n-1} e\left(\frac{Sp^{\beta}x^2}{p^n}\right) &= \frac{1}{p^{\alpha}} \sum_{x=0}^{p^n-1} e\left(\frac{Sx^2}{p^{n-\beta}}\right) \sum_{y=0}^{p^{\alpha}-1} e\left(\frac{(x-w)y}{p^{\alpha}}\right)\notag\\
            &= \frac{1}{p^{\alpha}} \sum_{y=0}^{p^{\alpha}-1} e\left(\frac{-wy}{p^{\alpha}}\right) \sum_{x=0}^{p^n-1} e\left(\frac{Sx^2}{p^{n-\beta}} + \frac{xy}{p^{\alpha}}\right)\notag\\
            &= \frac{p^{\beta}}{p^{\alpha}} \sum_{y=0}^{p^{\alpha}-1} e\left(\frac{-wy}{p^{\alpha}}\right) \sum_{x=0}^{p^{n-\beta}-1} e\left(\frac{Sx^2+ xy p^{n-\alpha-\beta}}{p^{n-\beta}}\right),
        \end{align}
where we have used (1.4) to extract $p^{\beta}$. Observe that as $2\alpha+\beta \leq n$, we have $\ds \left(p^{n-\alpha-\beta}\right)^2 \equiv 0 \imod
{p^{n-\beta}}$. By completing the square modulo $p^{n-\beta}$ we obtain
    \begin{align}
        Sx^2 + p^{n-\alpha-\beta}xy \equiv S(x+(2S)^{-1} p^{n-\alpha-\beta}xy)^2 \imod {p^{n-\beta}}.
    \end{align}
As $x$ runs over a complete residue system modulo $p^{n-\beta}$, so does the bracketed expression in (2.3).
Thus, (2.2) simplifies to
    \begin{align}
        \sum_{\substack{x=0\\x \equiv w \imod {p^{\alpha}}}}^{p^n-1} e\left(\frac{Sp^{\beta}x^2}{p^n}\right) &= \frac{p^{\beta}}{p^{\alpha}}\cdot G(S;p^{n-\beta}) \sum_{y=0}^{p^{\alpha}-1} e\left(\frac{-wy}{p^{\alpha}}\right).
    \end{align}
The innermost sum of (2.4) is a geometric sum. As $w \not\equiv 0 \imod {p^{\alpha}}$, the right-hand side of (2.4) will reduce to zero, which
completes the first part of the  lemma.

We now suppose that $p=2$ and $2\alpha+\beta \leq n -2$.
We proceed in a similar manner as in the odd prime case to arrive at
    \begin{align}
        \sum_{\substack{x=0\\x \equiv w \imod {2^{\alpha}}}}^{2^n-1} e\left(\frac{Sx^2}{2^n}\right) &= \frac{2^{\beta}}{2^{\alpha}} \sum_{y=0}^{2^{\alpha}-1} e\left(\frac{-wy}{2^{\alpha}}\right) \sum_{x=0}^{2^{n-\beta}-1}
e\left(\frac{Sx^2+2^{n-\alpha-\beta}xy}{2^{n-\beta}}\right).
    \end{align}
As $2\alpha + \beta \leq n - 2$, we have $\ds \left(2^{n-\alpha-\beta-1}\right)^2 \equiv 0 \imod {2^{n-\beta}}$.
Completing the square as before, we obtain
    \begin{align}
        Sx^2 + 2^{n-\alpha-\beta}xy \equiv S\left(x+S^{-1}2^{n-\alpha-\beta-1}y\right)^2 \imod {2^{n-\beta}}.
    \end{align}
Hence, with (2.6) we see that (2.5)  simplifies to
    \begin{align*}
       \sum_{\substack{x=0\\x \equiv w \imod {2^{\alpha}}}}^{2^n-1} e\left(\frac{Sx^2}{2^n}\right)= \frac{2^{\beta}}{2^{\alpha}}\cdot G(S;2^{n-\beta}) \sum_{y=0}^{2^{\alpha}-1} e\left(\frac{-wy}{2^{\alpha}}\right),
    \end{align*}
which simplifies to the second part of the lemma.
\end{proof}

\begin{lemma}  
    Let $p$ be a prime, $n \in \nn$, $w \in \zz$ and $S \in \zz$ coprime to $p$.
Let $\alpha, \beta \in \nn_0$ satisfy $\alpha, \beta \leq n$.

    If $p$ is odd and $2\alpha+\beta \leq n$, we have
        \begin{align*}
            \sum_{x=0}^{p^n-1} e\left(\frac{Sp^{\beta}(p^{\alpha}x+w)^2}{p^n}\right) = \begin{cases}
                0 &\mbox{if } w \not\equiv 0 \imod {p^{\alpha}}\\
                p^{\alpha+\beta} \cdot G(S;p^{n-\beta}) &\mbox{if } w \equiv 0 \imod {p^{\alpha}}.
            \end{cases}
        \end{align*}

    If $p=2$ and $2\alpha+\beta \leq n -2$, we have
        \begin{align*}
            \sum_{x=0}^{2^n-1} e\left(\frac{S2^{\beta}(2^{\alpha}x+w)^2}{2^n}\right) = \begin{cases}
                0 &\mbox{if } w \not\equiv 0 \imod {2^{\alpha}}\\
                2^{\alpha+\beta} \cdot G(S;2^{n-\beta}) &\mbox{if } w \equiv 0 \imod {2^{\alpha}}.
            \end{cases}
        \end{align*}
\end{lemma}

\begin{proof}
    For any prime $p$, from (1.4) we have
        \begin{align*}
            \sum_{x=0}^{p^n-1} e\left(\frac{Sp^{\beta}(p^{\alpha}x+w)^2}{p^n}\right) &= p^{\beta} \sum_{x=0}^{p^{n-\beta}-1} e\left(\frac{S(p^{\alpha}x+w)^2}{p^{n-\beta}}\right).
        \end{align*}
    Hence, we may replace $n$ by $n-\beta$ in the statements of the lemma to arrive at the same result.
Thus, we may suppose that $\beta = 0$. We have
        \begin{align}
            \sum_{x=0}^{p^n-1} e\left(\frac{S(p^{\alpha} x+w)^2}{p^n}\right) = \sum_{\substack{x=0\\x \equiv w \imod {p^{\alpha}}}}^{p^{n+\alpha}-1} e\left(\frac{Sx^2}{p^n}\right) = p^{\alpha} \sum_{\substack{x=0\\x \equiv w \imod {p^{\alpha}}}}^{p^{n}-1} e\left(\frac{Sx^2}{p^n}\right).
        \end{align}
    The assertions of the lemma follow by considering the parity of $p$ in (2.7) with respect to Lemma~2.3.
\end{proof}

We now prove a simple result regarding the double gauss sum $G(a,b,c;S;p^n)$.

    \begin{lemma}  
        Let $p$ be a prime, $n \in \nn$ and $S \in \zz$  coprime to $p$.
Let $b \in \zz$ and write $b \equiv p^{\beta}B \imod {p^n}$ for $\beta \in \nn_0$ and $B\in \zz $  coprime to $p$. Then we have
        \begin{align*}
            G(0,b,0;S;p^n) &= p^{n+\beta}.
        \end{align*}
\end{lemma}
\begin{proof}
    We have
        \begin{align}
            G(0,b,0;S;p^n) &= \sum_{x,y=0}^{p^n-1} e\left(\frac{p^{\beta}Bxy}{p^n}\right).
        \end{align}
    The result is clear when $\beta =n$ and so we may assume that $\beta < n$.
Along with (1.5), the expression for $G(0,b,0;S;p^n)$ in (2.8) will simplify to
        \begin{align*}
       p^{2\beta} \sum_{x,y=0}^{p^{n-\beta}-1} e\left(\frac{Bxy}{p^{n-\beta}}\right)= p^{2\beta}\left[ p^{n-\beta}+ \sum_{y=1}^{p^{n-\beta}-1} \sum_{x=0}^{p^{n-\beta}-1} e\left(\frac{Bxy}{p^{n-\beta}}\right)\right] = p^{n+\beta},
        \end{align*}
which completes the proof.
\end{proof}

We now diagonalize our binary quadratic form $Q$ given in (1.6).

\begin{theorem}  
Let $Q$ be the binary quadratic form given in {\em (1.6)}. If $a \neq 0$, we have
\begin{align}
4a \cdot Q = (2ax+by)^2 + \DD y^2.
\end{align}
\end{theorem}
\begin{proof}
Let
$\ds M := \left(\begin{array}{cc}
2a & b\\
b & 2c\end{array}\right)$. Then
$\ds Q = [x\quad y]\frac M2 [x \quad y]^{T}$.
As $a \neq 0$, we can express the matrix $M$ as $\ds
M = LDL^T$,
where $\ds L = \left(\begin{array}{cc}
1 & 0\\
\frac{b}{2a} & 1\end{array}\right)$ and $\ds D = \left(\begin{array}{cc}
2a & 0\\
0 & \frac{\DD}{2a}\end{array}\right)$. Therefore, we have
\begin{align}
    Q = [x \quad y] L \frac D2 L^T [x \quad y]^T = \left([x \quad y] L \right) \frac D2 \left([x \quad y] L \right)^T.
\end{align}
Let $\ds X_1:= x+ \frac{b}{2a} y$ and $X_2 := y$ denote the change of variables given by $L$. We substitute this into (2.10) to obtain
\begin{align*}
Q =[X_1 \hspace{5pt} X_2] \frac D2 [X_1 \hspace{5pt} X_2]^T
= a X_1^2 + \frac{\Delta}{4a}X_2^2 = \frac{1}{4a}(2aX_1)^2 + \frac{\Delta}{4a}X_2^2,
\end{align*}
which completes the proof.
\end{proof}

Note that one can easily expand the square and collect like terms in (2.9) to show equality. Our $LDL^T$ diagonalization method would generalize to a quadratic form
in $n$ variables, with a certain non-singularity condition. It's well known that any integral quadratic form is equivalent to a diagonal quadratic form with rational
coefficients \cite[pp. 69-70]{LED1}, and so our method may work for other diagonalizations of $Q$.

We look at the equation in (2.9) modulo a prime power. Allowing for the divisibility of $4a$, one can deduce the following corollary from Theorem 2.2.

\begin{coro} 
Let $Q$ be the quadratic form given in {\em (1.6)} and suppose $a \neq 0$.
Let $p$ be a prime and $n \in \nn$.
We write $a = p^{\alpha}A$ for $\alpha \in \nn_0$ and $A \in \zz$ coprime to $p$.

If $p$ is odd, we have
\begin{align*}
p^{\alpha} \cdot Q \equiv (4A)^{-1}(2Ap^{\alpha}x+by)^2 + (4A)^{-1}\Delta y^2 \imod {p^{n+\alpha}}.
\end{align*}

If $p=2$, we have
    \begin{align*}
    2^{\alpha+2} \cdot Q \equiv A^{-1}(2^{\alpha+1}Ax+by)^2 + A^{-1}\Delta y^2 \imod {2^{n+\alpha+2}}.
    \end{align*}
\end{coro}

\section{Explicit Evaluation of Double Gauss Sums}

\begin{theorem}  
Let $p$ be a prime, $n \in \nn$, $S \in \zz$ coprime to $p$ and let $a, b, c \in \zz$ be as in {\em (1.9)}, namely
\begin{align*}
        (a,b,c)=1,\quad  a\equiv p^{\alpha}A \imod {p^n},\quad  b \equiv p^{\beta}B \imod {p^n},\quad  c \equiv 0 \imod {p^{\alpha}},
    \end{align*}
where $A, B \in \zz$ satisfy  $p \nmid AB$, and $\alpha , \beta \in \nn_0$. Note that if $\ds p^{n} \mid a$, then by convention we set $\alpha = n$ and $A = 1$. We
set $\DD := 4ac-b^2$.

If $p$ is odd, we have
\begin{align*}
G(a,b,c;S;p^n)  = G(SA;p^{n-\alpha}) \cdot G(SA\DD;p^{n+\alpha}).
\end{align*}

If $p=2$, we have
\begin{align*}
G(a,b,c;S;2^n) = \begin{cases}
        \ds \frac 14 \cdot G(SA;2^{n-\alpha}) \cdot G(SA\DD;2^{n+\alpha+2}) &\mbox{if } \alpha +1 < n\\
        (-1)^c (1+(-1)^{b+1}) &\mbox{if } \alpha+1\geq n=1\\
        2^n &\mbox{if } \alpha + 1 \geq n > 1.
    \end{cases}
\end{align*}
\end{theorem}

\begin{proof}
We first suppose that $p$ is odd. By (1.2), (1.5) and Corollary 2.2, we have
    \begin{align}
        G(a,b,c;S;p^n) &= \frac{1}{p^{2\alpha}} G(a,b,c;Sp^{\alpha};p^{n+\alpha})\notag\\
         &\hspace{-10mm}= \frac 1{p^{2\alpha}} \sum_{x,y=0}^{p^{n+\alpha}-1} e\left(\frac{S \left( (4A)^{-1}(2Ap^{\alpha}x+by)^2 + (4A)^{-1} \DD y^2\right)}{p^{n+\alpha}}\right)\notag\\
        &\hspace{-10mm}= \frac{1}{p^{2\alpha}} \sum_{y=0}^{p^{n+\alpha}-1} e\left(\frac{S(4A)^{-1}\DD y^2}{p^{n+\alpha}}\right) \sum_{x=0}^{p^{n+\alpha}-1} e\left(\frac{S(4A)^{-1}(2Ap^{\alpha}x + p^{\beta}By)^2}{p^{n+\alpha}}\right).
    \end{align}

Assume that $\alpha \leq \beta$. We have
\begin{align*}
\ds 2Ap^{\alpha}x+p^{\beta}y \equiv p^{\alpha}(2Ax+p^{\beta-\alpha}By)\imod {p^{n+\alpha}}.
\end{align*}
As $x$ runs over a complete residue system modulo $p^{n+\alpha}$, so does $\ds 2Ax+p^{\beta-\alpha}By$ for any fixed value of $y$. Therefore, with (1.4) and Lemma
2.1, (3.1) will simplify to the statement of the theorem.

Assume now that $\beta < \alpha$ and in particular this means $\beta = 0$ which in turn implies $(\DD,p)=1$. As $\alpha \leq n$, by Lemma 2.4, we see that the
innermost sum of (3.1) is non-zero if and only if $y \equiv 0 \imod {p^{\alpha}}$. In such an instance, along with Lemma~ 2.1, the sum indexed by $x$ in (3.1) will
be given by $\ds p^{\alpha} \cdot G(SA;p^{n+\alpha})$. Hence, by modifying the index of the outermost sum in (3.1) by this congruence condition, the double sum in
(3.1) will simplify to
    \begin{align}
        \frac 1{p^{\alpha}} \cdot G(SA;p^{n+\alpha}) \sum_{\substack{y=0\\y \equiv 0 \imod {p^{\alpha}}}}^{p^{n+\alpha}-1} e\left(\frac{S(4A)^{-1} \DD y^2}{p^{n+\alpha}}\right).
    \end{align}
Observe that by Corollary 2.1, we have
    \begin{align}
        \frac 1{p^{\alpha}} G(SA;p^{n+\alpha}) &= \left(\frac {SA}p\right)^{n+\alpha} \imath^{\left(\frac{p^{n+\alpha}-1}2\right)^2} p^{\frac{n-\alpha}2}=\left(\frac {SA}p\right)^{n-\alpha} \imath^{\left(\frac{p^{n-\alpha}-1}2\right)^2} p^{\frac{n-\alpha}2}\notag\\
         &= G(SA;p^{n-\alpha}).
    \end{align}
As $\alpha \leq n$, we may use Lemma 2.2 to simplify the sum indexed by $y$ in (3.2). Thus,  with (3.3) and Lemma 2.1, (3.2) will simplify to $\ds G(SA;p^{n-\alpha})
\cdot G(SA\DD;p^{n+\alpha})$.

We now suppose that $p=2$. If $\alpha = n$, then $\beta = 0$ so that $b$ is odd. Additionally, $c$ is even, and so the statement of the theorem will agree with Lemma
2.5 for all $n \geq 1$. Thus, we may assume without loss of generality that $\alpha < n$. Similar to the above, with (1.2) and Corollary 2.2, we deduce that
$G(a,b,c;S;2^n)$ is given by
    \begin{align}
    \frac{1}{2^{2(\alpha+2)}} \sum_{y=0}^{2^{n+\alpha+2}-1} e\left(\frac{SA^{-1}\DD y^2}{2^{n+\alpha+2}}\right)\sum_{x=0}^{2^{n+\alpha+2}-1} e\left(\frac{SA^{-1}\left(2^{\alpha+1}Ax+2^{\beta}By\right)^2}{2^{n+\alpha+2}}\right).
    \end{align}
    If $\alpha +1 \leq \beta$, then we can extract a common factor of $2^{\alpha+1}$ as before so that with (1.4) and Lemma 2.1, (3.4) will simplify to
        \begin{align*}
            \frac{1}{2^{2(\alpha+2)}} &\cdot G(SA\DD;2^{n+\alpha+2}) \cdot G(SA2^{2(\alpha+1)};2^{n+\alpha+2})\notag\\
             &= \frac 14 \cdot G(SA\DD;2^{n+\alpha+2}) \cdot G(SA;2^{n-\alpha}).
        \end{align*}
    Suppose instead that $\beta < \alpha+1$, so that we have $\beta = 0$ and $\DD$ odd. We look to evaluate the innermost sum in (3.4). If $\alpha + 2 \leq n$, by Lemma 2.4, we have that (3.4) simplifies to
    \begin{align}
        \frac 1{2^{\alpha+3}} \cdot G(SA;2^{n+\alpha+2}) \sum_{\substack{y=0\\y \equiv 0 \imod {2^{\alpha}}}}^{2^{n+\alpha+2}-1} e\left(\frac{SA^{-1}\DD y^2}{2^{n+\alpha+2}}\right).
    \end{align}
Subsequently, as $\alpha < n$, by Lemma 2.2, (3.5) will simplify to
    \begin{align}
        \frac 1{2^{\alpha+3}}\cdot G(SA;2^{n+\alpha+2}) \cdot G(SA\DD;2^{n+\alpha+2}) = \frac 14 \cdot G(SA;2^{n-\alpha})\cdot G(SA\DD;2^{n+\alpha+2}),
    \end{align}
where we have simplified with Corollary 2.1 and Lemma 2.1 as necessary.

Thus, suppose now that $\alpha + 1 = n$ and observe that we cannot use Lemma 2.4 in this case. If $n=1$, then $a$ is odd and hence
    \begin{align*}
        G(a,b,c;S;2) &= \sum_{x,y=0}^1 e\left(\frac{S(ax+bxy+cy)}2\right) = 1 + (-1)^a + (-1)^c + (-1)^{a+b+c}\notag\\
            &= (-1)^c (1+ (-1)^{b+1}),
    \end{align*}
which agrees with the statement of the theorem.
Otherwise, if $\alpha+1 = n > 1$ we have $\beta = 0$, so that
    \begin{align}
        G(a,b,c;S;2^n) &= \sum_{x,y=0}^{2^n-1} e\left(\frac{S(2^{n-1}Ax^2+Bxy+cy^2)}{2^n}\right)\notag\\
            &= \sum_{y=0}^{2^n-1} e\left(\frac{Scy^2}{2^n}\right) \sum_{x=0}^{2^n-1} (-1)^x e\left(\frac{SBxy}{2^n}\right) .
    \end{align}
The sum indexed by $x$ in (3.7) can be written as
    \begin{align}
        \sum_{\substack{x=0\\x \text{ even}}}^{2^n-1} e\left(\frac{SBxy}{2^n}\right) - \sum_{\substack{x=0\\x \text{ odd}}}^{2^n-1} e\left(\frac{SBxy}{2^n}\right) &= 2 \sum_{\substack{x=0\\x \text{ even}}}^{2^n-1} e\left(\frac{SBxy}{2^n}\right) - \sum_{x=0}^{2^n-1} e\left(\frac{SBxy}{2^n}\right)\notag\\
    &= 2\sum_{\substack{x=0\\y \equiv 0 \imod {2^{n-1}}}}^{2^{n-1}-1} 1 - \sum_{\substack{x=0\\y \equiv 0 \imod {2^n}}}^{2^n-1} 1.
    \end{align}
By assumption, we have $2^{\alpha} \mid c$ so we may write $c \equiv c_1 2^{\alpha} \imod {2^n}$ for some integer $c_1$. Hence, $\ds e\left(\frac{Scy^2}{2^n}\right)
= (-1)^{c_1y}$. Thus, together with (3.8), breaking up the sum in (3.7) according to the parity of $y$ yields
    \begin{align}
         G(a,b,c;S;2^n) = \sum_{\substack{y=0\\y \text{ even}}}^{2^n-1}
& \Bigg( 2\sum_{\substack{x=0\\y \equiv 0 \imod {2^{n-1}}}}^{2^{n-1}-1} 1 - \sum_{\substack{x=0\\y \equiv 0 \imod {2^n}}}^{2^n-1} 1\Bigg)\notag\\
    &+ (-1)^{c_1} \sum_{\substack{y=0\\y \text{ odd}}}^{2^n-1} \Bigg( 2\sum_{\substack{x=0\\y \equiv 0 \imod {2^{n-1}}}}^{2^{n-1}-1} 1 - \sum_{\substack{x=0\\y \equiv 0 \imod {2^n}}}^{2^n-1} 1\Bigg).
    \end{align}
As $n \geq 2$, the second term of (3.9) vanishes, and we're left with
    \begin{align*}
         G(a,b,c;S;2^n) &= \sum_{y=0}^{2^{n-1}-1} \Bigg( 2\sum_{\substack{x=0\\y \equiv 0 \imod {2^{n-2}}}}^{2^{n-1}-1} 1 - \sum_{\substack{x=0\\y \equiv 0 \imod {2^{n-1}}}}^{2^n-1} 1\Bigg)\notag\\
 &= 2 \sum_{y=0}^{2^{n-1}-1} \sum_{\substack{x=0\\y \equiv 0 \imod {2^{n-1}}}}^{2^{n-1}-1} 1 - \sum_{y=0}^{2^{n-1}-1} \sum_{\substack{x=0\\y \equiv 0 \imod {2^{n-1}}}}^{2^{n-1}-1} 1\notag\\
    &= 2^{n+1}-2^n = 2^n,
    \end{align*}
which agrees with the statement of the theorem.
\end{proof}

We observe that for $p$ odd, Theorem 3.1 agrees with Lemma 2.5 if $\alpha = n$. First, note that $\alpha > 0$ implies $\beta = 0$ and so $\DD$ is coprime to $p$.
Thus, by Corollary 2.1, the statement of Theorem 3.1 simplifies to
    \begin{align*}
        G(SA;p^{n-\alpha}) \cdot G(SA\DD; p^{n+\alpha}) = G(SA\DD;p^{2n}) = p^n.
    \end{align*}
Continuing in this manner, we have the following corollary.

\begin{coro}  
Let $p$ be a prime, $n \in \nn$, $S \in \zz$ coprime to $p$ and let $a, b, c \in \zz$ be as in {\em (1.9)}, namely
\begin{align*}
        (a,b,c)=1,\quad  a\equiv p^{\alpha}A \imod {p^n},\quad  b \equiv p^{\beta}B \imod {p^n},\quad  c \equiv 0 \imod {p^{\alpha}},\quad
    \end{align*}
where $A, B \in \zz$ satisfy  $p \nmid AB$,  and $\alpha , \beta \in \nn_0$. We recall our convention that if $\ds p^{n} \mid a$, then we set $\alpha = n$ and $A =
1$. We also  recall that $\DD := 4ac-b^2$. We write
\begin{align*}
\begin{cases}
\DD \equiv p^{\delta} D \imod {p^{n+\alpha}} &\mbox{if $p$ is odd}\\
\DD \equiv 2^{\delta}D \imod {2^{n+\alpha+2}} &\mbox{if $p=2$},
\end{cases}
\end{align*}
where $D \in \zz$ satisfies $p \nmid D$, and $\delta \in \nn_0$.

If $p$ is odd, we have
        \begin{align*}
                G(a,b,c;S;p^n) = p^{n+\frac{\delta}2} \left(\frac{SA}p\right)^{\delta}\left(\frac Dp\right)^{n+\alpha+\delta}\left(\frac{-1}p\right)^{(n+\alpha)(\delta+1)} i^{\left(\frac{p^{\delta}-1}2\right)^2}.
        \end{align*}

If $p=2$ and  $\alpha+1 < n$, we have
    \begin{align*}
        G(a,b,c;S;2^n) = \begin{cases}
            \ds 2^{\frac{3n+\alpha}2}\left(\frac 2{SA}\right)^{\delta}(1+\imath^{SA}) &\mbox{if } \delta = n+\alpha+2\\
            \ds 0 &\mbox{ if } \delta = n+\alpha+1\\
            \ds 2^{n+\frac{\delta}2} \left(\frac 2{SA}\right)^{\delta}\left(\frac 2D\right)^{n+\alpha+\delta} \imath^{SA\left(\frac{D+1}2\right)^2} &\mbox{if } \delta \leq n+\alpha.
        \end{cases}
    \end{align*}

If $p=2$ and  $\alpha+1 \geq n$, we have
    \begin{align*}
G(a,b,c;S;2^n) = \begin{cases}
        2^n &\mbox{if } \alpha + 1 \geq n > 1\\
        (-1)^c (1+(-1)^{b+1}) &\mbox{if } \alpha+1\geq n=1.
    \end{cases}
\end{align*}
\end{coro}

\begin{proof}
    Suppose $p$ is odd. From Theorem 3.1, (1.4) and Corollary 2.1, we have
        \begin{align}
            G(a,b,c;S;p^n) &= G(SA;p^{n-\alpha}) \cdot G(SA p^{\delta}D;p^{n+\alpha})\notag\\
            &= \begin{cases}
                p^{n+\alpha} \cdot G(SA;p^{n-\alpha}) &\mbox{if } \delta = n+\alpha\\
                p^{\delta} \cdot G(SA;p^{n-\alpha}) \cdot G(SAD;p^{n+\alpha-\delta}) &\mbox{if } \delta < n+\alpha
            \end{cases}\notag\\
            &= \begin{cases}
                p^{\frac{3n+\alpha}2}\left(\frac {SA}p\right)^{n+\alpha} i^{\left(\frac{p^{n+\alpha}-1}2\right)^2} &\mbox{if } \delta = n+\alpha\\
                p^{n+\frac{\delta}2} \left(\frac{SA}p\right)^{\delta} \left(\frac Dp\right)^{n+\alpha+\delta}i^{\left(\frac{p^{n+\alpha}-1}2\right)^2}i^{\left(\frac{p^{n+\alpha+\delta}-1}2\right)^2} &\mbox{if } \delta < n+\alpha.
            \end{cases}
        \end{align}
    Observe that the cases in (3.10) will agree when $\delta = n+\alpha$. Subsequently, we have
        \begin{align}
            i^{\left(\frac{p^{n+\alpha}-1}2\right)^2}i^{\left(\frac{p^{n+\alpha+\delta}-1}2\right)^2} 
&= \left.\begin{cases}
            \left(\frac{-1}p\right)^{n+\alpha} &\mbox{if } \delta \text{ even}\\
            i^{\left(\frac{p^{\delta}-1}2\right)^2} &\mbox{if } \delta \text{ odd}
            \end{cases}\right. \\
   &= \left(\frac{-1}p\right)^{(n+\alpha)(\delta+1)} i^{\left(\frac{p^{\delta}-1}2\right)^2}. \nonumber
        \end{align}
    Hence, with (3.10) and (3.11) we may deduce the statement of the corollary.

    Suppose now $p=2$ and $\alpha+1< n$. Note that this implies $n \geq 2$. By Theorem~ 3.1, (1.4) and Corollary 2.1, we have
        \begin{align}
            G(a,b,c;S;2^n) &= \frac 14 \cdot G(SA;2^{n-\alpha}) \cdot G(SA2^{\delta}D;2^{n+\alpha+2})\notag\\
            &= \begin{cases}
                2^{n+\alpha} \cdot G(SA;2^{n-\alpha}) &\mbox{if } \delta = n+\alpha+2\\
                0 &\mbox{if } \delta = n +\alpha+1\\
                2^{\delta-2} \cdot G(SA;2^{n-\alpha}) \cdot G(SAD;2^{n+\alpha+2-\delta}) &\mbox{if } \delta \leq n+\alpha
            \end{cases}\notag\\
            &= \begin{cases}
                2^{\frac{3n+\alpha}2} \left(\frac 2{SA}\right)^{n+\alpha}(1+\imath^{SA}) &\mbox{if } \delta = n+\alpha+2\\
                0 &\mbox{if } \delta = n+\alpha+1\\
                2^{n+\frac{\delta}2-1} \left(\frac 2{SA}\right)^{\delta} \left(\frac 2D\right)^{n+\alpha+\delta}(1+i^{SA})(1+i^{SAD}) &\mbox{if } \delta \leq n+\alpha.
            \end{cases}
    \end{align}
    Depending on the residue class of $D$ modulo $4$, we have
        \begin{align}
            (1+i^{SA})(1+i^{SAD}) = \begin{cases}
                2\imath^{SA} &\mbox{if } D \equiv 1 \imod 4\\
                2 &\mbox{if } D \equiv 3 \imod 4.
            \end{cases}
        \end{align}
    Thus, with (3.12) and (3.13) we may deduce the statement of the corollary.

Finally the case $p=2$ and $\alpha+1 \geq n$ follows  immediately from Theorem 3.1.
\end{proof}

We note that the results of Theorem 3.1 and Corollary 3.1 agree with the results of Weber \cite[p. 22]{Weber} and Alaca, et al. \cite[pp. 129-132]{AAW}.

\section{Examples}  

We provide a few examples to illuminate our method.

\begin{example}  
Suppose that   $Q:=Q(x,y) = x^2 + xy + y^2$, so that $a=b=c=1$.
Thus $\ds M=\left(\begin{array}{cc}
    2 & 1\\
    1 & 2\end{array}\right)$
is the symmetric integral matrix associated with $Q$, and  $\DD = 3$.

For $p > 3$, we take $A = 1$, $D=3$ and $\alpha = \delta = 0$ in {\em Corollary 3.1}  to obtain
    \begin{align}
        G(1,1,1;S;p^n) = p^n \left(\frac {-3}p\right)^n.
    \end{align}

For $p=3$, we take $A=1$, $D=1$, $\alpha = 0$  and $\delta = 1$ in {\em Corollary 3.1}  to obtain
    \begin{align}
        G(1,1,1;S;3^n) = 3^{\frac{2n+1}2} \left(\frac S3\right) i.
    \end{align}

Finally, for $p=2$, we assume for the sake of discussion that $n \geq 2$. 
We take $A=1$, $D = 3$ and $\alpha = \delta = 0$ in {\em Corollary 3.1} to obtain 
    \begin{align}
        G(1,1,1;S;2^n) = 2^n \left(\frac 23\right)^{n} =(-1)^n 2^n.
    \end{align}
Alternatively, we can use {\em Theorem 3.1} and subsequently {\em Theorem 2.1} to obtain the same results.
\end{example}

We note that (4.1) and (4.2) are given by Corollary 2.1(i) \cite[p. 137]{AAW},
and (4.3) agrees with Corollary 3.2(i) \cite[p. 151]{AAW}  of Alaca, et al.

\begin{example}  
Suppose that $Q:= Q(x,y)=3x^2 + xy + 3y^2$, so that $a = c= 3$ and $b = 1$. Thus
$\ds M= \left(\begin{array}{cc}
    6 & 1\\
   1 & 6 \end{array}\right)$ is the symmetric integral matrix associated with $Q$, and  $\DD = 35$.

For $p > 7$, we take $A = 3$, $D = 35$ and  $\alpha = \delta = 0$ in {\em Corollary 3.1}  to obtain
    \begin{align*}
        G(3,1,3;S;p^n) = p^n \left(\frac{-35}p\right)^n.
    \end{align*}

For $p=7$, we take $A = 3$, $D = 5$, $\alpha = 0$ and $\delta = 1$  in {\em Corollary 3.1} to obtain
    \begin{align*}
        G(3,1,3;S;7^n) = 7^{\frac{2n+1}2} \left(\frac{3S}7\right) \left(\frac 57\right)^{n+1} i
        &= (-1)^n 7^n \left(\frac S7\right) \sqrt{-7}.
    \end{align*}

For $p=5$, we take  $A = 3$, $D = 7$, $\alpha = 0$ and $\delta = 1$  in {\em Corollary 3.1}  to obtain
    \begin{align*}
        G(3,1,3;S;5^n) = 5^{\frac{2n+1}2} \left(\frac{3S}5\right) \left(\frac 75\right)^{n+1} 
            &= (-1)^n 5^n \left(\frac S5\right) \sqrt{5}.
    \end{align*}

Finally, for $p=2$, and assuming $n \geq 2$, we take  $A = 3$, $D = 35$ and $\alpha = \delta = 0$ in {\em Corollary 3.1}
 and we conclude that
    \begin{align*}
        G(3,1,3;S;2^n) = 2^n \left(\frac 2{35}\right)^n = (-1)^n 2^n.
    \end{align*}
Alternatively, we can use  {\em Theorem 3.1} and {\em Theorem 2.1} to obtain the same results.
\end{example}

\section{Remarks}  

It is straightforward to deduce the following theorem using our method.
\begin{theorem}
    Let $k \in \nn$ be odd, and $a, b, c \in \zz$ be as in {\em (1.9)}. 
If $(a\DD,k)=1$, then
    \begin{align*}
        G(a,b,c;S;k) = \left(\frac{-\DD}k\right) \cdot k.
    \end{align*}
\end{theorem}
\begin{proof}
    Using the approach of Theorem 3.1, we obtain
    \begin{align*}
        G(a,b,c;S;k) &= G(Sa;k) \cdot G(Sa\DD;k) \\
  &= \left(\frac {Sa}k\right)\imath^{\left(\frac{k-1}2\right)^2} k^{\frac 12} \cdot \left(\frac {Sa\DD}k\right)\imath^{\left(\frac{k-1}2\right)^2} k^{\frac 12}\notag\\
        &= \left(\frac{-\DD}k\right) \cdot k,
    \end{align*}
which completes the proof.
\end{proof}

We plan to show in an upcoming paper how we may use our method to give an explicit evaluation of a quadratic form Gauss sum in $n$ variables.
We also plan on demonstrating how we may use our evaluation of the double Gauss sum  $G(a,b,c;S;p^n)$
to determine an explicit formula for  the number of solutions to the congruence $\ds ax^2+bxy+cy^2 \equiv k \imod {p^n}$ for a given integer $k$.

\section*{Acknowledgments}
The research of \c{S}aban Alaca was supported
by a  Discovery Grant from the Natural Sciences and Engineering Research Council of Canada (RGPIN-2015-05208).

\vspace{3mm}
\noindent
\c{S}aban Alaca and Greg Doyle \\
School of Mathematics and Statistics \\
Carleton University, Ottawa \\
Ontario, K1S 5B6, Canada \\

\noindent
SabanAlaca@cunet.carleton.ca\\
gdoyle@math.carleton.ca

\end{document}